\documentclass[a4paper,11pt]{amsart}

\newfont{\cyr}{wncyr10 scaled 1100}

\usepackage[left=2.7cm,right=2.7cm,top=3.5cm,bottom=3cm]{geometry}
\usepackage{amsthm,amssymb,amsmath,amsfonts,mathrsfs,amscd,graphics}
\usepackage[latin1]{inputenc}
\usepackage[all]{xy}
\usepackage{latexsym}
\usepackage{longtable}
\usepackage{color}
\usepackage{stmaryrd}
\usepackage{graphicx}

 \usepackage[cal=euler]{mathalpha}

\usepackage{marvosym}
\usepackage{amstext}
\usepackage{etoolbox}
\pretocmd\mvchr{\text}{}{\errmessage{Patching \noexpand\mvchr failed}}
\pretocmd\textmvs{\text}{}{\errmessage{Patching \noexpand\textmvs failed}}

\theoremstyle{plain}
\newtheorem{theorem}{Theorem}[section]

\newtheorem{lemma}[theorem]{Lemma}
\newtheorem{proposition}[theorem]{Proposition}
\newtheorem{conjecture}[theorem]{Conjecture}

\theoremstyle{definition}
\newtheorem{definition}[theorem]{Definition}

\theoremstyle{remark}
\newtheorem{obswr}[theorem]{Observation}
\newtheorem{remarkwr}[theorem]{Remark}

\newenvironment{remark}{\begin{remarkwr}\begin{upshape}}{\end{upshape}\end{remarkwr}}

\newcommand{\sD}{{\mathscr{D}}}

\newcommand{\bb}{\mathbb}

\newcommand{\scr}{\mathscr}
\newcommand{\mrm}{\mathrm}
\newcommand*{\bfcdot}{\scalebox{0.7}{$\bullet$}}

\newcommand{\cS}{{\mathcal S}}
\newcommand{\cT}{{\mathcal T}}

\newcommand{\cA}{\mathcal A}

\newcommand{\cW}{\mathcal W}

\newcommand{\cH}{\mathcal H}
\newcommand{\cC}{\mathcal C}

\newcommand{\cE}{\mathcal E}

\newcommand{\Q}{\mathbb{Q}}

\newcommand{\Z}{\mathbb{Z}}

\newcommand{\F}{\mathbb{F}}

\newcommand{\C}{\mathbb{C}}

\newcommand{\PP}{\mathbb{P}}

\newcommand{\GL}{\mathrm{GL}}

\newcommand{\ord}{{\mathrm{ord}}}
\newfont{\gotip}{eufb10 at 12pt}

\newcommand{\cO}{{\mathcal O}}

\newcommand{\lra}{\longrightarrow}

\newcommand{\NS}{\mathcal{Q}}

\newcommand{\SL}{{\mathrm {SL}}}

\newcommand{\R}{{\mathbb R}}

\newcommand{\res}{\mathrm{res}}

\def\lra{{\longrightarrow}}
\def\SL{{\bf SL}}
\def\GL{{\bf GL}}

\def\rp1{r\!\!+\!\!1}

\def\XXint#1#2#3{{\setbox0=\hbox{$#1{#2#3}{\int}$}%
\vcenter{\hbox{$#2#3$}}\kern-.5\wd0}}%

\include{thebibliography}

\begin{document}

\title[Mock plectic points]{Mock plectic points}

\author{Henri Darmon  and Michele Fornea}

\dedicatory{To Jan Nekov\'a$\check{\text{r}}$}

\begin{abstract}
A $p$-arithmetic subgroup of $\SL_2(\Q)$  like the Ihara group $\Gamma := \SL_2(\Z[1/p])$  acts    by M\"obius transformations on the   Poincar\'e upper half plane
$\cH$ and on  Drinfeld's $p$-adic upper half plane $\cH_p := \PP_1(\C_p)-\PP_1(\Q_p)$.  The diagonal action of $\Gamma$ on the product is discrete, and 
the quotient $\Gamma\backslash(\cH_p\times \cH)$ can be envisaged as a ``mock Hilbert modular surface". 
According to a  striking prediction of Nekov\'a$\check{\text{r}}$ and Scholl, the 
CM points on genuine Hilbert modular surfaces should give rise to ``plectic Heegner points" that encode non-trivial regulators attached, notably,  to elliptic curves of rank two over real quadratic fields.
This article  develops the analogy between Hilbert modular surfaces  and their mock counterparts, with the aim of transposing    the plectic philosophy  
 to the mock Hilbert setting, where the analogous 
plectic invariants are expected to lie in the alternating square of the Mordell--Weil group of certain elliptic curves of rank two over $\Q$.
     \end{abstract}

\address{H. D.: McGill University, Montreal, Canada}
\email{darmon@math.mcgill.ca}
\address{M.F.: CRM, Barcelona}
\email{mfornea.research@gmail.com}
 

\maketitle

\tableofcontents

\section*{Introduction}
The Hilbert modular group 
$\SL_2(\cO_F)$
 attached to  a real quadratic field $F$, viewed as  a discrete subgroup of 
 $\SL_2(\R)\times \SL_2(\R)$ by  ordering the  real 
  embeddings 
  $\nu_1,\nu_2\colon F\to \R$,
 acts  discretely by M\"obius transformations  on the product
$\cH\times\cH$ of two Poincar\'e  upper half planes. The  cohomology 
of the complex surface 
$$ \cS_F:= \SL_2(\cO_F)\backslash (\cH\times \cH) $$
 is  intimately tied with the arithmetic of elliptic curves with everywhere
 good reduction over $F$.
More precisely, if
 $E_{/F}$ is such a (modular) elliptic curve,
$ E_j :=E\otimes_{F,\nu_j}\R$ are the associated real elliptic curves for $j=1,2$,
  and  $\pi_E$ is the associated automorphic representation of $\GL_2(F)$, Oda's period conjecture predicts  an isomorphism 
\begin{equation}
\label{eqn:oda}
H^2(\cS_F, \Q)[\pi_E] \simeq
 H^1(E_1,\Q)\otimes H^1(E_2,\Q) 
 \end{equation}
 of rational Hodge structures \cite{oda}.  
 This strong ``geometric" form of modularity
 has implications  for the arithmetic of $E_{/F}$ that  are richer, more subtle and less well understood than those that arise  from realising $E$ as a quotient of the jacobian of a Shimura curve. 
 For instance, let 
   $K\subset \mathrm{M}_2(F)\subset M_2(\R) \times M_2(\R)$ be a 
   quadratic extension
of $F$ which is 
  ``Almost Totally Real" (ATR), i.e.,
   satisfies 
   \[
   K\otimes_{F,\nu_1} \R = \C, \qquad K\otimes_{F,\nu_2} \R =\R \oplus \R.
   \]
 Let 
  $\tau_1\in \cH$  be the fixed point for 
   the action of $\nu_1(K^\times)\subset \GL_2(\R)$
    on $\PP_1(\C)$, and let
 $\tau_2,\tau_2'\in \R$ be the fixed points of
 $\nu_2(K^\times)$.  Denote by
  $(\tau_2,\tau_2')$  the hyperbolic geodesic in $\cH$ joining $\tau_2$ to $\tau_2'$,  and let
    $\gamma \subset \cS_F$
 be the simple  closed geodesic  contained in the image of 
 $$\{\tau_1\} \times (\tau_2,\tau_2')\subset \cH\times \cH.$$  The finiteness of $H_1(\cS_F,\Z)$ 
 implies there there is an integer $m\ge1$ and a smooth real  two-dimensional region
  $\Pi\subset \cS_F$ having   $m\gamma$ as its boundary. 
 Oda's period conjecture     \eqref{eqn:oda} 
 implies that, for a suitable
  real analytic two-form  $\omega \in \Omega^2(\cS_F)[\pi_E]$,
   the complex integral  
\begin{equation}
\label{eqn:logan}
P_\gamma :=  \frac{1}{m}\int_{\Pi} \omega \  \in \ \C 
  \end{equation}
   is  independent of the choice of  $\Pi$ 
  up to elements in a suitable  period lattice  $\Lambda_{1}$ attached to $E_1$.
Viewing  \eqref{eqn:logan} as an element 
    of  $\C/\Lambda_{1} \cong E_1(\C)$, 
the complex point  $P_\gamma$  is conjectured to be  defined  over an explicit  ring class field  of $K$, following a  numerical recipe that is worked out and tested experimentally in \cite{darmon-logan}  and \cite{ComputATR}.

 \medskip
 Analogously, the {\em Ihara group}  $\Gamma := \SL_2(\Z[1/p])$  
acts  by M\"obius transformations
on  $\cH$ and on   Drinfeld's $p$-adic
 upper half plane $\cH_p:= \PP_1(\C_p) \setminus \PP_1(\Q_p)$. 
Its diagonal action on the product $\cH_p\times \cH$ is discrete, and 
the quotient 
$$\cS:= \Gamma\backslash(\cH_p\times\cH)$$ can be envisaged as a ``mock Hilbert modular surface", following a suggestive terminology of Barry Mazur \cite{mazur-pc}.
Fleshing out the    analogy between $\cS_F$ and 
$\cS$ leads to  fruitful   perspectives on   the arithmetic of elliptic curves (and modular abelian varieties) over $\Q$ with multiplicative reduction at $p$. 
Notably, 
\begin{itemize}
\item [$\bfcdot$]
the ``exceptional zero conjecture" on derivatives of the $p$-adic $L$-functions of these elliptic curves formulated by  Mazur, Tate and Teitelbaum \cite{mtt} and proved by Greenberg and Stevens \cite{gs} 
can be understood   as
the counterpart of  \eqref{eqn:oda} for $\cS$;
\item[$\bfcdot$]
the mock analogues  of the ATR points  of 
\eqref{eqn:logan}
are the Stark--Heegner points of \cite{darmon-hpxh} which are indexed by real quadratic geodesics on $\cS$ and are conjecturally defined over ring class fields of real quadratic fields. 
\end{itemize}
These two analogies  are briefly explained in Sections 
\ref{sec:mttgs} and \ref{sec:stark-heegner} respectively.

\medskip
 A striking insight of Nekov\'a$\check{\text{r}}$ and Scholl 
(\cite{HiddenSym}, \cite{NekRubinfest}, \cite{nekovar-scholl})
suggests 
 that zero-dimensional CM cycles  on $\cS_F$ 
should   give rise to ``plectic Heegner points" involving  non-trivial regulators for elliptic curves (over $F$) of rank two.  At the moment, no precise numerical recipe is available to compute them, placing the conjectures of loc.cit.~somewhat outside
 the scope of experimental verification. (The reader is nevertheless  referred to \cite{PlecticJacobians} for some results in that direction.)
More recently,  the second author and Lennart Gehrmann have explored the implications of 
 the plectic conjectures in the setting of quaternionic Shimura varieties uniformised by products of Drinfeld's upper half planes \cite{fornea-gehrmann}. They constructed unconditionally $p$-adic realizations of plectic Heegner points which admit a concrete analytic description. In that context the plectic philosophy has been tested   experimentally in \cite{fgm},  and some partial evidence has been given in \cite{polyquadraticPlectic}, even though its theoretical underpinnings remain  poorly understood. 
  These non-archimedean perspectives  suggest  that it might be instructive to examine the  
 plectic philosophy in the  intermediate setting of mock Hilbert modular surfaces,
whose periods  involve a somewhat delicate mix  of complex and $p$-adic integration.
 
 \medskip
 The primary aim of this note is to develop  the analogy between $\cS_F$ and $\cS$  and describe its most important arithmetic applications, 
with a  special emphasis on 
 the plectic framework where it had not been  examined systematically before. 
 The  main new contribution, presented in Section  
\ref{sec:plectic}, is the construction of  \emph{global} cohomology classes -- referred to as ``mock plectic invariants" -- associated to  elliptic curves over $\Q$ of conductor $p$ and  CM points on $\cS$. The techniques developed in this paper  can be used to generalize and upgrade the construction of plectic $p$-adic invariants of  \cite{fgm}, \cite{fornea-gehrmann} in the CM setting,   relating these objects to the leading terms  of anticyclotomic $p$-adic $L$-functions that were introduced and studied in \cite{bd-Mumford}. Viewed in this way, the plectic philosophy is seen to be consistent with the  anticyclotomic Birch and Swinnerton-Dyer conjecture  in the somewhat exotic setting of loc.cit., where the twisted $L$-values that one wishes to interpolate {\em vanish identically} and it becomes necessary to
$p$-adically interpolate the  Heegner points themselves over the anticyclotomic tower, viewed as 
algebraic avatars of  {\em first derivatives} of  twisted $L$-series.
The authors hope that this  consistency provides some oblique evidence for the  plectic philosophy  of Nekov\'a$\check{\text{r}}$ and Scholl, while enriching the dictionary between Hilbert modular surfaces and their mock counterparts.

\bigskip\noindent
{\bf Acknowledgements}.
The authors are grateful for the hospitality of the  Mathematical Sciences Research Institute in Berkeley during the Spring of 2023
when this project was initiated. They also thank Lennart Gehrmann and Matteo Tamiozzo for stimulating discussion surrounding the topics of this paper. The first author was supported by an NSERC Discovery grant, a Simons fellowship, and a  Clay Senior Scholarship during that time. The second author was supported by the National Science Foundation under Grant No. DMS-1928930 while in residence at the MSRI.
 
\bigskip
This work owes much to the   broad vision and deep insights
of Jan Nekov\'a$\check{\text{r}}$. The authors dedicate it to him, with admiration and gratitude. 
\bigskip

\section{Mock Hilbert modular forms and their periods}
\label{sec:mttgs}

A {\em mock Hilbert modular form} on $\cS$ (of parallel weight $2$) 
should be thought of,   loosely speaking,   as a ``holomorphic differential two-form on $\Gamma\backslash(\cH_p\times \cH)$", i.e., a function 
$f(z_p, z_\infty)$ of the variables $z_p\in \cH_p$ and $z_\infty\in \cH$ which is rigid analytic in $z_p$, holomorphic in $z_\infty$, and satisfies the transformation rule
\begin{equation}
\label{eqn:fake-object}
``  f\left( \frac{az_p + b}{c z_p+d}, \frac{az_\infty + b}{c z_\infty+d}\right) =  (cz_p + d)^2 (c z_\infty + d)^2 f(z_p, z_\infty)"  \qquad \mbox{ for all } \left(\begin{array}{cc} a & b \\ c &  d \end{array}\right) \in \Gamma.
 \end{equation}
 The  awkward mix of rigid  and complex analysis inherent in this  (non) definition  prevents 
 \eqref{eqn:fake-object}
 from resting on a solid mathematical foundation. A few simple facts about rigid differentials on the Drinfeld upper half plane can nonetheless   be made to conjure a concrete object that captures key features of
  \eqref{eqn:fake-object}.  
 
 \subsection{Digression: rigid analytic differentials on $\cH_p$}
 
 Let $\cT := \cT_0 \sqcup \cT_1$ denote
  the Bruhat--Tits tree   of 
 $\SL_2(\Q_p)$, whose set   $\cT_0$ of vertices is 
  in bijection with homothety classes of 
 $\Z_p$-lattices in $\Q_p^2$, two lattices being joined by an edge in 
 $\cT_1\subset \cT_0^2$  if they are represented by lattices contained one in the other with index $p$. 
There is a  natural 
 {\em reduction map}
 $$ {\mathrm {red}}\colon \cH_p \longrightarrow \cT $$
   from $\cH_p$  to $\cT$, which maps 
   the {\em standard affinoid}
  \begin{equation}
  \label{eqn:Av0}
   \cA_{\circ} := \{ z \in  \cO_{\C_p} \mbox{ such that } |z-a|\ge 1, \mbox{ for all }  a\in \Z_p \} \  \subset \ \PP_1(\C_p)
   \end{equation}
to  the  vertex attached to the lattice
  $v_\circ = [\Z_p^2]$.
  The $p+1$ mod $p$ residue discs in the complement of $\cA_{\circ}$   are in natural bijection with $\PP_1(\F_p)$  and contain the boundary annuli
 \begin{eqnarray}
 \label{eqn:Aej}
 \cW_{\infty} &=& \{ z\in \PP_1(\C_p) \mbox{ such that } 1< |z|<p    \}, 
 \\ \nonumber
   \cW_{j} &=&  \{ z\in \PP_1(\C_p) \mbox{ such that } 1/p< |z-j|<1 \},
    \mbox{ for } j=0,\ldots, p-1.
 \end{eqnarray}
    The edges having $v_\circ$ as an endpoint
are likewise in bijection with $\PP_1(\F_p)$ 
  by setting
\begin{equation}
\label{eqn:distedges}
e_\infty \leftrightarrow (\Z_p^2,  \  \Z_p\cdot (1,0) + p \Z_p^2), \qquad 
 e_j \leftrightarrow (\Z_p^2,  \ \Z_p\cdot (j,1) + p \Z_p^2), \ \  \mbox{ for } j=0,\ldots, p-1.
\end{equation}
  The preimage of the  singleton $\{e_j\}$    under the reduction map is the annulus $\cW_j$, for each $j\in \PP_1(\F_p)$.
The  properties 
$$ {\rm red}^{-1}(\{v_\circ\}) = \cA_{\circ}, 
\qquad {\rm red}^{-1}(\{e_j\}) = \cW_j, \ \ \mbox{ for all } j\in \PP_1(\F_p),$$
combined with the requirement of compatibility with the natural actions of $\SL_2(\Q_p)$ on 
 $\cH_p$ and on $\cT$, determine  the reduction map uniquely.
  In particular, 
the preimage, denoted $\cA_v$, 
  of the vertex  $v\in \cT_0$ is an  affinoid  obtained by  taking the complement in $\PP(\C_p)$ of    $(p+1)$ mod $p$ residue discs with $\Q_p$-rational centers (relative to  a coordinate on $\PP(\Q_p)$ depending on $v$).
 Given an edge $e=(v_1,v_2) \in \cT_1$, the affinoids  $\cA_{v_1}$ and $\cA_{v_2}$ 
 are glued to each other along the 
$p$-adic  annulus attached to $e$,  denoted $\cW_e$. With just a modicum of artistic licence, the entire Drinfeld upper-half plane can  be visualised  as a tubular neighbourhood of $\cT,$  as in the figure  below for $p=2$.

\bigskip
\begin{figure}[ht!]
\centering
\includegraphics[width=90mm]{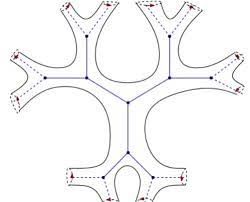}
\caption{The Drinfeld upper half plane and the Bruhat-Tits tree \label{fig:Hp}}
\end{figure}

\bigskip\noindent
This picture  suggests that $\cH_p$, unlike its Archimedean counterpart, is   far from being simply connected and 
that its first cohomology is quite rich. 
For each edge $e\in \cT_1$,
the  de Rham cohomology of  $\cW_e$ is identified with $\C_p$ 
 via the map that sends $\omega \in \Omega^1_{\mathrm {rig}}(\cW_e)$
 to its $p$-adic annular residue, denoted ${\mathrm {res}}_{\cW_e}(\omega)$.  This residue map is well-defined up to a sign, which is determined by fixing an {\em orientation} on $\cW_e$, or, equivalently, viewing $e$ as an {\em ordered} edge of $\cT$, having a source and target. Let $\cE(\cT)$ denote the set of such ordered edges,    let $s,t:\cE(\cT) \rightarrow  \cT_0$ denote the source and target maps,
 and write $\bar e$ for  the edge $e$ with its source and target interchanged.
  
 \begin{definition}
 A {\em harmonic cocycle} on $\cT$ is a $\C_p$-valued function 
 $$ c:  \cE(\cT) \rightarrow \C_p$$
  satisfying the following properties:
 \begin{itemize}
 \item[$\bfcdot$]  $ c(\bar e) =-c(e)$,   for all  $e\in \cE(\cT)$;
  \item[$\bfcdot$] for all vertices $v$ of $\cT$,  
 $$\sum_{s(e)=v}  c(e) =   \sum_{t(e)=v} c(e)=0.$$ 
  \end{itemize}
  \end{definition} 
 The $\C_p$-vector space of $\C_p$-valued harmonic cocycles on $\cT$ is denoted ${\mathrm C}_{har}(\cT, \C_p)$.
The class of a rigid analytic differential $\omega\in \Omega^1_{\mathrm {rig}}(\cH_p)$ in the de Rham cohomology of $\cH_p$ is encoded in the
$\C_p$-valued function $c_\omega$ on $\cE(\cT)$ defined by
 $$ c_\omega(e) = {\mathrm {res}}_{\cW_e}(\omega).$$
 That $c_\omega$ is  a harmonic cocycle follows directly from  the residue theorem for rigid differentials. 
  The oriented edges of $\cT$ are also in natural bijection with the compact open balls in $\PP_1(\Q_p)$, by assigning to 
  $e\in \cE(\cT)$ the ball $U_e$ 
 according to the following 
 prescriptions:
 $$ U_{\bar e} \sqcup U_e = \PP_1(\Q_p), \qquad U_{e_\infty} = \PP_1(\Q_p) - \Z_p, \qquad U_{\gamma e} = \gamma U_e,  
 \mbox{ for all } \gamma \in \Gamma,$$
 where $e_\infty$ is the distinguished edge of 
 $\cE(\cT)$ evoked in \eqref{eqn:distedges}.
 The harmonic cocycle 
 $c_\omega$ can therefore be parlayed into a $\C_p$-valued distribution $\mu_\omega$ satisfying the defining property
 $$ \mu_\omega(U_e) = c_\omega(e),$$
 where $U_e\subset \PP_1(\Q_p)$ is the open ball corresponding to the ordered edge $e$.
 The distribution $\mu_\omega$ lives in the dual space of locally constant $\C_p$-valued functions on $\PP_1(\Q_p)$.
 
   \medskip
 For this paragraph, and this paragraph only, let $\Gamma\subset \SL_2(\Q_p)$ be a group acting {\em discretely} on $\cH_p$ and for which the quotient graph $\Gamma \backslash \cT$ is finite. If the rigid differential $\omega$ is $\Gamma$-invariant, the harmonic cocycle $c_\omega$ takes on finitely many values and is therefore
 {\em $p$-adically bounded}.   The distribution $\mu_\omega$ then extends to a $\C_p$-valued measure, which can be integrated against continuous functions on $\PP_1(\Q_p)$. 
 The differential $\omega$ can then be obtained from 
 $\mu_\omega$ by the rule
\begin{equation}
\label{eqn:teitelbaum}
 \omega = \int_{\PP_1(\Q_p)} \frac{\mathrm{d}z}{z-t}\hspace{1mm} \mathrm{d}\mu_\omega(t),
 \end{equation}
 a special case of Jeremy Teitelbaum's $p$-adic Poisson Kernel formula
 \cite{teitelbaum}
  which  recovers a rigid analytic modular form on $\cH_p$ from its associated boundary distribution.
 
 \subsection{Modular form-valued harmonic cocycles as mock residues}
 
 Returning to the setting where $\Gamma= \SL_2(\Z[1/p])$   and to the   dubious
notion of a mock Hilbert modular form  on $\Gamma\backslash (\cH_p\times\cH)$  proposed in  \eqref{eqn:fake-object}, the discussion in the previous section suggests at least
 what  its system of $p$-adic annular residues ought to look like:
 \begin{definition}
 \label{def:mock-residue}
 A   {\em system of mock residues} is a  harmonic cocycle
 $$ c : {\cE}(\cT) \longrightarrow \Omega^1({\cH})$$
 with values in the space $\Omega^1({\cH})$
 of holomorphic differentials  on $\cH$,
 satisfying 
 \begin{itemize}
 \item[$\bfcdot$]  $c(e)$ is a weight two    cusp form  on   the stabiliser $\Gamma_e$ of $e$ in $\Gamma$, i.e., a holomorphic differential on the standard compactification of $\Gamma_e\backslash \cH$;
 \item [$\bfcdot$]
 more generally, for all $\gamma\in \Gamma$ and all $e \in \cE(\cT)$, 
 $$ \gamma^* c(\gamma e) = c(e).$$
 \end{itemize}
 \end{definition}
 \noindent
 Roughly speaking, a system of  mock residues is what one might expect to obtain from the $p$-adic annular residues of 
 a mock Hilbert modular form of parallel weight two. But unlike 
 \eqref{eqn:fake-object}, Definition 
 \ref{def:mock-residue} is completely rigorous.
Since 
 $\Gamma$ acts transitively on the  unordered  edges of $\cT$, 
 and   because
   the Hecke congruence group
 $\Gamma_0(p)$   is the
    stabiliser in $\Gamma$
    of  the distinguished edge 
  $e_\infty\in \cE(\cT)$ of \eqref{eqn:distedges},  
   the map $c \mapsto c(e_\infty)$ identifies the   complex vector 
   space $\cC_{\rm har}(\cT,\Omega^1(\cH))^{\Gamma}$ of mock residues  for $\Gamma$
 with  the space $S_2(\Gamma_0(p))^{p\mbox{-}\mathrm{new}}$ of weight two newforms of level $p$.
It transpires that 
 mock Hilbert modular forms -- or at least,   their systems of
 mock $p$-adic   residues --
   are  merely a slightly overwrought
      incarnation of classical modular forms of weight two on $\Gamma_0(p)$.

 \subsection{$\C$-valued distributions}
 \label{sec:distributions}
 Given a $p$-new weight two cusp form $f$ on  $\Gamma_0(p)$,  denote by 
 $c_f$  the associated mock residue, and 
 write $f_e:= c_f(e) \in S_2(\Gamma_e)$.
For any $x,y\in \cH^* := \cH \sqcup \PP_1(\Q)$, 
 the  assignment 
\begin{equation}
\label{eqn:C-dist}
 e \mapsto  2\pi i\int_x^y f_e(z)\hspace{0.5mm} \mathrm{d}z 
 \end{equation}
 is a $\C$-valued harmonic  cocycle on $\cT$, denoted $c_f[x,y]$. It  determines a $\C$-valued distribution  $\mu_f[x,y]$
 on $\PP_1(\Q_p)$, which can be integrated against locally constant
 complex-valued functions on $\PP_1(\Q_p)$.  
 In order to 
integrate $\mu_f[x,y]$  against Teitelbaum's $p$-adic Poisson kernel
as in \eqref{eqn:teitelbaum},   the   distribution 
  $\mu_f[x,y]$ needs to  be upgraded to  a   measure with suitable integrality properties.
 
 \subsection{Modular symbols}
 \label{sec:modular-symbols}
 Suppose henceforth that $f$ is a Hecke eigenform with rational fourier coefficients, and let $E_{/\bb{Q}}$ denote the corresponding strong Weil curve.
  The theory of modular symbols shows that the values of the harmonic  cocycle 
  $c_f[x,y]$ acquire good integrality properties  when $x,y$ belong to the boundary $\PP_1(\Q)$ of the extended upper half plane. More precisely, Manin and Drinfeld have shown that the values 
\begin{eqnarray*}
 c_f[r,s](e) &=&  2\pi i\int_r^s f_e(z) \hspace{0.5mm} \mathrm{d}z  \qquad r,s\in \PP_1(\Q)   \\
 &=& 2\pi i\int_{\gamma r}^{\gamma s} f(z) \hspace{0.5mm} \mathrm{d}z, 
 \qquad \mbox{where } \gamma\in \Gamma \mbox{  satisfies }  \gamma e = e_{\infty}   
 \end{eqnarray*}  
 belong to a lattice $\Lambda_f \subset \C$ which is commensurable with the period lattice $\Lambda_E$ of $E$.
 Restricting the function $(x,y) \mapsto c_f[x,y]$
 to $\PP_1(\Q)\times \PP_1(\Q)$ leads to a modular symbol with values in the space
  of $\Lambda_f$-valued harmonic cocycles on $\cT$. For economy of notation, the resulting $\Lambda_f$-valued measures on $\PP_1(\Q_p)$ will continue to be   denoted
 $$ \mu_f[r,s] \ \in  \ {\mathrm {Meas}}(\PP_1(\Q_p), \Lambda_f).$$
  The $\Lambda_f$-valued measures $\mu_f[r,s]$ are intimately connected to special values of the Hasse--Weil $L$-series attached to $E$, via the formulae 
 \begin{equation}
 \mu_f[0,\infty](\Z_p) = L(E,1), \qquad  \mu_f[0,\infty](\Z_p^\times) = (1-a_p(E)) \cdot L(E,1),
 \end{equation}
 where $a_p(E)= 1$  or $-1$ depending on whether $E$ has split or non-split multiplicative reduction at $p$.
 The Mazur--Swinnerton-Dyer $p$-adic $L$-function attached to $E$ (viewed as taking values in $\Q_p\otimes \Lambda_f$)
 is   the Mellin--Mazur transform  of $\mu_f[0,\infty]$ restricted to $\Z_p^\times$: 
\begin{equation}
\label{eqn:msd}
 L_p(E,s) = \int_{\Z_p^\times} \langle x\rangle^{s-1} d\mu_f[0,\infty](x).
 \end{equation}
 More generally, if $\chi$ is a primitive Dirichlet character of conductor $c$ prime to $p$, the twisted $L$-values 
 $L(E,\chi,1)$  can be obtained  analogously from the measures $\mu_f[a/c,\infty]$: 
 \begin{eqnarray*}
 \sum_{a\in (\Z/c\Z)^\times} \overline{\chi}(a) \cdot\mu_f[-a/c,\infty](\Z_p) &=&  \frac{c}{\tau(\chi)}\cdot L(E,\chi,1), \\
\sum_{a\in (\Z/c\Z)^\times} \overline{\chi}(a) \cdot\mu_f[-a/c,\infty](\Z_p^\times) &=& \frac{c}{\tau(\chi)}\cdot\big(1-\overline{\chi}(p) a_p(E)\big) \cdot  L(E,\chi,1),
 \end{eqnarray*}
 where $\tau(\chi)=\sum_{a\in(\Z/c\Z)^\times}\chi(a)e^{2\pi i \frac{a}{c}}$ is the Gauss sum attached to $\chi$,
while the Mazur--Swinnerton-Dyer $p$-adic $L$-function can be defined by setting
\begin{equation}
\label{eqn:msd-bis}
 L_p(E,\chi,s) := \sum_{a\in (\Z/c\Z)^\times} \chi(a) \int_{\Z_p^\times} \langle x\rangle^{s-1} d\mu_f[-a/c,\infty](x).
 \end{equation}
   Even more importantly for the constructions that will follow, 
a system of $\Lambda_f \otimes \C_p$-valued
 rigid differentials
 on the $p$-adic upper half-plane can also be obtained by integrating  the measures
 $\mu_f[r,s]$ against Teitelbaum's Poisson kernel:
\begin{equation}
\label{eqn:omegafrs}
 \omega_f[r,s] \ := \ \int_{\PP_1(\Q_p)} \frac{\mathrm{d}z}{z-t}\hspace{1mm} \mathrm{d}\mu_f[r,s](t)  \ \ \  \in  \ \ \Omega^1_{\mathrm{rig}}(\cH_p) \otimes \Lambda_f.
 \end{equation}
 The assignment $\omega_f: (r,s)\mapsto \omega_f[r,s]$ determines  a $\Gamma$-equivariant modular symbol with values in $\Omega^1_{\mathrm{rig}}(\cH_p) \otimes \Lambda_f$, satisfying
 $$ \gamma^* \omega_f[\gamma r,\gamma s]
 = \omega_f[r,s], \qquad \mbox{ for all } \gamma \in \Gamma.$$

 \subsection{The Mazur--Tate--Teitelbaum conjecture}
 Write $K_p$ for the quadratic unramified extension of $\Q_p$, let $\cA^\times$ denote the multiplicative group of nowhere vanishing rigid analytic functions on $\cH_p$ endowed with the $\Gamma$-action induced by M\"obius transformations, and let $\cA^\times/K_p^\times$ be the quotient by the subgroup of constant $K_p^\times$-valued functions.
 The logarithmic derivative $F \mapsto \mrm{d}F/F$ gives a $\Gamma$-equivariant map from $ \cA^\times/K_p^\times$ to  $\Omega^1_{\rm rig}(\cH_p)$, whose image contains the rigid differentials   $\omega_f[r,s]$: 
 \begin{lemma}
 \label{lemma:invert-dlog}
 The differentials $\omega_f[r,s]$ are in the image of the logarithmic derivative map, i.e., there are elements $F_f[r,s]\in (\cA^\times/K_p^\times) \otimes \Lambda_f$ satisfying
 \begin{equation}
 \label{eqn:dloglemma}
 {\rm dlog}\hspace{0.5mm}F_f[r,s] = \omega_f[r,s].
 \end{equation}
  \end{lemma}
  \begin{proof}
 A {\em  partitioning} of   $\PP_1(\Q_p)$ is a collection 
\begin{equation}
\label{eqn:defcC}
\cC = \{ (C_1,t_1), \ldots, (C_m,t_m)\},
\end{equation}
 where the $C_j$  are compact open subsets of
 $\PP_1(\Q_p)$ which are mutually disjoint  and
 satisfy 
 $$\PP_1(\Q_p) =  C_1 \sqcup \cdots \sqcup C_m, \qquad t_j\in C_j \mbox{ for } j=1,\ldots, m.$$
 The set of partitionings of $\PP_1(\Q_p)$ is   equipped with a natural partial ordering  in which 
 $ \cC \le \cC'$ if each of the compact open subsets involved in $\cC'$ is contained in one of the compact open subsets 
 arising  in $\cC$.
 Let $\mu$ be a 
 $\Z$-valued mesure of total measure zero on 
 $\PP_1(\Q_p)$.  Each partitioning of $\PP_1(\Q_p)$ 
  gives rise to a system of degree zero divisors on $\PP_1(\Q_p)$  by associating to 
  the partitioning $\cC$ of \eqref{eqn:defcC}
   the divisor
 $$ \sD_{\cC} := \sum_{j=1}^m \mu(C_j) \cdot [t_j]. $$
Fix a base point $z_0\in \cH_p(K_p)$ and let    
 $F_{\cC}$ be the unique  rational function satisfying
 $$ {\rm Divisor}(F_{\cC}) = \sD_{\cC}, \qquad F_{\cC}(z_0)= 1,$$
 which exists because the divisor $\sD_{\cC}$ is supported on $\PP_1(\Q_p)$.
 The limit
 $$ F_\mu := \lim_{\cC} F_{\cC} \in \cA^\times $$
 taken over any maximal chain in the set of partitionings, is a well-defined element of $\cA^\times $ which depends only on  $\mu$  and on the base point 
 $z_0$, and whose image in $\cA^\times/K_{p}^\times$ does not depend on the choice of   $z_0$.
  Extending  this construction to  $\Lambda_f$-valued measures in the obvious way and applying it to the measures
 $\mu_f[r,s]$, it
 is readily verified that  
 $$F_f[r,s]:= F_{\mu_f[r,s]} \in  (\cA^\times/K_{p}^\times)\otimes \Lambda_f$$
 satisfies \eqref{eqn:dloglemma}.
  \end{proof}
  The assignment 
 $(r,s) \mapsto F_f[r,s]$ 
 defines a $\Gamma$-invariant modular symbol with values in the $\Gamma$-module
 $(\cA^\times/K_p^\times)\otimes \Lambda_f$, i.e.,
 $$ F_f \in {\rm MS}\big((\cA^\times/K_p^\times)\otimes \Lambda_f\big)^\Gamma,$$
 where ${\rm MS}(\Omega)$ denotes the $\Gamma$-module of modular symbols with values in a $\Gamma$-module $\Omega$.
 The obstruction to lifting  $F_f$
  to  ${\rm MS}(\cA^\times\otimes \Lambda_f)^\Gamma$ is intimately tied with the $p$-adic uniformisation of 
the elliptic curve $E$  which  has multiplicative reduction at $p$. Namely,   let $q \in \Q_p^\times$ be the  $p$-adic  Tate   period of $E$.
The following theorem is a consequence of the conjecture of Mazur, Tate and Teitelbaum 
\cite{mtt} and its proof  by  Greenberg and  Stevens \cite{gs}:
 \begin{theorem}
 \label{thm:mttgs}
 There exists a lattice $\Lambda_f'\supset\Lambda_f$ such that the modular symbol    $F_f$       can be lifted to a $\Gamma$-invariant
      modular symbol with values in $(\cA^\times/q^\Z)\otimes\Lambda_f'$.
  \end{theorem}
 \noindent
\begin{proof}[Sketch of proof] 
The functor $\mrm{MS}(-)$ of modular symbols is exact, and taking $\Gamma$-cohomology gives the exact sequence 
\[\xymatrix{
 \mrm{MS}\big((\cA^\times/q^\Z)\otimes\Lambda_{f}\big)^\Gamma\ar[r]^\eta& \mrm{MS}\big((\cA^\times/K_p^\times)\otimes\Lambda_{f}\big)^\Gamma\ar[r]^-\delta& H^1\big(\Gamma, \mrm{MS}((K_p^\times/q^\Z)\otimes\Lambda_{f})\big)
}\]
where $\ker(\eta)$ is finite  because so is
the abelianization of $\Gamma$ (\cite{Serre}, II, 1.4). The obstruction to lifting $F_f$ to $\mrm{MS}\big((\cA^\times/q^\Z)\otimes\Lambda_{f}\big)^\Gamma$ is 
encoded by the class $c_f=\delta(F_f)$
which can be represented by the $1$-cocycle
\[
\tilde{c}_f(\gamma)=\frac{\gamma\cdot\tilde{F_f}}{\tilde{F_f}}
\]
where $\tilde{F_f}\in\mrm{MS}(\mathcal{A}^\times\otimes\Lambda_f)$ is any lift of $F_f$.
Let $\log_q$  be the branch  the $p$-adic logarithm satisfying $\log_q(q)=0$ which induces a map
 $K_p^\times/q^\Z \to K_p$ with finite kernel. At the cost of replacing the lattice $\Lambda_f$ with $\Lambda_f'=\frac{1}{t}\Lambda_f$ for some $t\in\Z$, the claim of the theorem reduces to the equality
\[
\mrm{log}_q(c_f)=0 
\]
in $H^1\big(\Gamma, \mrm{MS}(K_p\otimes\Lambda_{f}')\big)$, or equivalently to 
\[
\log(c_f)=\frac{\log(q)}{\mrm{ord}_p(q)}\cdot \mrm{ord}_p(c_f).
\]
Now, \cite[Corollary 3.3 $\&$ Lemma 3.4]{darmon-hpxh} imply that  $\mrm{ord}_p(c_f)$ is non-trivial and that 
the two classes $\log(c_f)$ and $\mrm{ord}_p(c_f)$ are proportional. The factor of proportionality is obtained by  producing a suitable  triple $(\gamma,r,s)\in\Gamma\times\bb{P}_1(\bb{Q})^2$ such that any $1$-coboundary $b$ for $\Gamma$ with values in $\mrm{MS}((K_p^\times/q^\Z)\otimes\Lambda_{f}')$ satisfies $b(\gamma)[r,s]=0$ and
\begin{equation}
\label{aim}
\log(c_f)(\gamma)[r,s]=\frac{\log(q)}{\mrm{ord}_p(q)}\cdot \mrm{ord}_p(c_f)(\gamma)[r,s]\qquad\&\qquad \mrm{ord}_p(c_f)(\gamma)[r,s]\not=0.
\end{equation}
Note that the first requirement is satisfied when $\gamma\in \Gamma$ fixes $r$ and $s$. The stabiliser in $\Gamma$ 
 of any pair $(r,s)\in \PP_1(\Q)^2$ is generated (up to torsion) by a 
 hyperbolic matrix $\gamma_{r,s}$ which has powers of $p$ as its eigenvalues 
 and fixes  the differential $\omega_f[r,s]$.
 The   multiplicative 
 period 
 $$ J_f[r,s]: = \tilde{c}_f(\gamma_{r,s})[r,s](z)= \frac{\tilde F_f[r,s](\gamma_{r,s}z)}{\tilde F_f[r,s](z)},
 \quad \mbox{satisfying} \quad  \log( J_f[r,s]) = \int_{z}^{\gamma_{r,s} z} \omega_f[r,s],$$
 does not depend on the base point 
 $z\in \cH_p(K_p)$ and belongs to $\bb{Q}_p^\times\otimes\Lambda_f'$ \cite[Prop.~2.7]{darmon-hpxh}.
   When $(r,s)= (0,\infty)$, the  period is related  
to  the   central critical value $L(E,1)$  and to  the first  derivative of the Mazur--Swinnerton-Dyer $p$-adic $L$-function $L_p(E,s)$
  attached to $E$ in \eqref{eqn:msd}:
 $$ \ord_p(J_f[0,\infty]) =   \delta_p(E) \cdot L(E,1), \qquad
 \log(J_f[0,\infty]) =   \delta_p(E) \cdot  L_p'(E,1),
 $$
 where
 $ \delta_p(E) = 1$ if $a_p(E)=1$ and $\delta_p(E)=0$ if
$a_p(E) =-1$ (cf.~\cite[\S 2.2 $\&$ 2.3]{darmon-hpxh}).  
As the Mazur--Tate--Teitelbaum conjecture asserts that 
  $$ L_p'(E,1) = \frac{\log(q)}{\ord_p(q)}\cdot L(E,1),$$
  we deduce that 
  \[
\log(J_f[0,\infty])=\frac{\log(q)}{\ord_p(q)}\cdot\mrm{ord}_p(J_f[0,\infty]).
  \]
More generally, the valuations and logarithms of the periods $J_f[\infty,a/c]$ 
with $\gcd(a,c)=1$  can be expressed 
  in terms the special values  (resp.~derivatives) of partial $L$-series (resp.~partial $p$-adic $L$-series) whose linear combinations give all the twisted values
$L(E,\chi,1)$ and  $L_p'(E,\chi,1)$  defined in 
\eqref{eqn:msd-bis}, as  $\chi$ ranges over all primitive  Dirichlet characters  of conductor $c$ for which $\chi(p)= a_p(E)$ \cite[\S 2.2 $\&$ 2.3]{darmon-hpxh}. 
The Mazur--Tate--Teitelbaum conjecture for these $L$-series   and the non-vanishing result of \cite[Lemma 2.17]{darmon-hpxh}     then imply that \eqref{aim} can be achieved.
 \end{proof}
 
 Explicitly, Theorem \ref{thm:mttgs} ensures the existence of a lattice $\Lambda_f'\supset \Lambda_f$ such that $F_f$ is a $\Gamma$-invariant 
 modular symbol with values in $(\cA^\times/q^\Z)\otimes \Lambda_f'$ satisfying  
 
  \begin{itemize}
   \vspace{1mm}
 \item[$\bfcdot$] ${\rm dlog}\hspace{0.5mm}F_f[r,s] = \omega_f[r,s]$, for all $r,s\in \PP_1(\Q)$;

 \vspace{1mm}
 \item[$\bfcdot$] $ F[ \gamma r, \gamma s](\gamma z) = F[r,s](z) \pmod{q^{\Z}\otimes \Lambda_f'}, \quad \mbox { for all } \gamma\in \Gamma, \  r,s\in \PP_1(\Q), \mbox{ and } z\in \cH_p.$
  \vspace{1mm}
 \end{itemize}
The statement that the (multiplicative)
periods $J_f[r,s]$ of the ``mock Hilbert modular form" attached to $E$ lie in a lattice commensurable with  
 $q^\Z\otimes\Lambda_E$  resonates with Oda's period conjecture for Hilbert modular surfaces.
 The emergence of the Tate period in what had, up to now, been a rather formal sequence of constructions provides the first inkling that the point of view of mock Hilbert modular surfaces  
 opens genuinely   
 new perspectives on arithmetic questions related to $f$ and  its associated elliptic curve $E$.

\section{Stark--Heegner points}
\label{sec:stark-heegner}
A  {\em real multiplication} (RM)  point on $\cH_p$ is an element   $\tau\in\cH_p$ which also lies in a real quadratic field $K$. 
 Its {\em associated order} is the subring
\begin{equation}
\label{eqn:deforder}
 \cO_\tau := \left\{ \left(\begin{array}{cc} a& b \\ c & d \end{array}\right)\in \mathrm{M}_2\big(\Z[1/p]\big)  \mbox{ satisfying } a\tau + b = c\tau^2 + d\tau   \right\}
 \end{equation}
 of the matrix ring $\mathrm{M}_2\big(\Z[1/p]\big)$.  This
 commutative ring  is identified with
  a $\Z[1/p]$-order in  $K$ by sending a matrix in
  \eqref{eqn:deforder}
  to its automorphy factor $c\tau+d$. 
   Global class field theory associates to any $\Z[1/p]$-order $\cO\subset K$ an abelian extension $H_{\cO}$ (resp.~$H_{\cO}^+$)  of $K$
   whose Galois  group  over $K$
    is identified with the Picard group (resp.~the narrow Picard group) of  projective $\cO$-modules (resp.~of projective $\cO$-modules endowed with an orientation at $\infty$):
    $$ {\mathrm {Gal}}(H_{\cO}/K) = {\mathrm {Pic}}(\cO), 
    \qquad 
    {\mathrm {Gal}}(H_{\cO}^+/K) = {\mathrm {Pic}}^+(\cO).
    $$
The stabiliser in $\Gamma$  of the RM point   $\tau\in\cH_p$ is  identified with the group of norm one elements  in  $\cO_\tau$. Since $p$ is non-split in $K=\Q(\tau)$,   this stabiliser
 is of rank one. The choice of a fundamental unit of $K$, which is fixed once and for all,  determines a generator $\gamma_\tau$     of the stabiliser of $\tau$ modulo torsion. 
 The {\em Stark--Heegner point} attached to $\tau$ is the 
 element
$$ P_\tau := F_f[r, \gamma_\tau r](\tau) \in (\C_p^\times/q^\Z) \otimes \Lambda_f' = E(\C_p)\otimes \Lambda_f'.$$
The definition of $P_\tau$ ostensibly rests on the choice of an auxiliary base point $r\in \PP_1(\Q)$ but is ultimately independent of that choice.
After choosing real and imaginary generators $\Omega_f^+$ and $\Omega_f^-$ of $\Lambda_f'\cap \R$ and $\Lambda_f'\cap i\R$  and writing
$$ P_\tau = P_\tau^+ \cdot \Omega_f^+ + 
P_\tau^-\cdot \Omega_f^-,$$
the invariants $P_\tau^+$ and $P_\tau^- \in E(\C_p)$ 
are conjectured to satisfy the following \cite{darmon-hpxh}:
\begin{conjecture}
\label{conj:shp}
 The points  $P_\tau^+$ and $P_{\tau}^-$
are  defined over   the ring class field 
 $H_{\cO_\tau}$ and  the narrow  ring class field 
 $H_{\cO_\tau}^+$ respectively.
 The point $P_\tau^-$ is in the minus part for the action of complex conjugation
  on $E(H_{\cO_\tau}^+)$.
 \end{conjecture}
 The points $P_\tau^\pm$ are expected to 
 behave in most key respects just like classical Heegner  points over ring class fields of imaginary quadratic fields; in particular they should satisfy 
 an analogue of the Gross--Zagier formula.  Stark--Heegner points are the ``mock" counterpart of the ``ATR points" on elliptic curves over real quadratic fields, arising from topological one-cycles on a genuine Hilbert modular surface, that were alluded to in equation
    \eqref{eqn:logan} of the introduction.
The   properties of the points $P_\tau^\pm$ 
predicted in Conjecture \ref{conj:shp}  are poorly
  understood, just as they are for their    ATR counterparts, in spite of the theoretical 
 evidence  obtained in \cite{RationalitySH}, \cite{Longo-Vigni}, \cite{Longo-Martin-Yan},  and \cite{AsterisqueBDRSV} for instance.

\begin{remark}
 The construction of Stark--Heegner points has been generalised to various
 different settings 
   over the years, notably in 
   \cite{Trif}, \cite{Greenberg}, \cite{ArbitraryDarmon}, \cite{AutomorphicDarmon},  and \cite{fornea-gehrmann}.
\end{remark}

Conjecture \ref{conj:shp} is consistent with the Birch and Swinnerton-Dyer conjecture, since the sign in the functional equation for $L(E/K,s)$ is always $-1$ for $E$ an elliptic curve (or modular abelian variety) of conductor $p$ and $K$ a real quadratic field in which $p$ is inert. The same is true as well for the $L$-functions $L(E/K,\chi,s)$ twisted by ring class characters $\chi$ of prime-to-$p$ conductor. It follows that
$$ \ord_{s=1} L(E/K,\chi,s) \ge 1, \qquad \mbox{ for all } \chi:{\rm Gal}(H_{\cO_\tau}/K) \lra \C^\times, $$
and hence that 
$$ \ord_{s=1} L(E/H_\tau,s) \ge [H_{\cO_\tau}:K].$$
The Stark--Heegner point construction gives a conjectural analytic recipe for   the systematic supply of
non-trivial global points over ring class fields of $K$  whose existence is predicted by the Birch and Swinnerton-Dyer conjecture.

\section{Mock plectic invariants}
\label{sec:plectic}
A remarkable insight of Nekov\'a$\check{\text{r}}$ and Scholl suggests 
 that zero-dimensional CM cycles  on 
 Hilbert modular surfaces
should   encode determinants of global points for elliptic curves of rank two called ``plectic Heegner points". 
This suggests that the CM points on $\cH_p\times \cH$ are just as interesting arithmetically as the RM points on $\cH_p$ that lead to Stark--Heegner points.
The goal of this last chapter is to 
describe the ``mock plectic invariants" attached to CM zero-cycles on  the mock Hilbert surface $\cS$  and to explore the  relevance of these invariants  for the arithmetic of elliptic curves of rank two.
   
Let   $K$ be   a quadratic imaginary field, 
viewed simultaneously as a subfield of $\C_p$ and $\C$, and embedded diagonally in $\C_p\times \C$. 
A   point $\tau = (\tau_p,\tau_\infty) \in (\cH_p\times \cH)\cap K $
  is called a {\em CM point}
on $\cS$ attached to $K$.
For simplicity, it shall be assumed henceforth that its  associated order, defined as in 
\eqref{eqn:deforder}, 
 is the maximal $\Z[1/p]$-order in the imaginary quadratic field $K$,  that this order has class number one,  and that the prime $p$ is inert in $K$, leaving aside the slightly more delicate case where $p$ is ramified.
 
In contrast with the setting for Conjecture \ref{conj:shp} and the discussion following it,
  the sign in the functional equation for $L(E/K,\chi,s)$ is now systematically equal to $1$, for  any ring class character $\chi$  of $K$  of prime-to-$p$ conductor. The
  Birch and Swinnerton-Dyer conjecture therefore predicts that $E(K)$ has {\em even rank}.
A systematic supply of Heegner points 
  over $K$ or over ring class fields of conductor prime to $p$
 is therefore not expected  to arise in this setting.
Rather, the mock plectic invariant attached to $\tau$ 
will  be used to prove the implication
$$  L(E/K,1)\ne 0  \ \ \Rightarrow \ \ E'(\bb{Q}) \mbox{ is finite,} $$
where $E'_{/\bb{Q}}$ is the quadratic twist of $E$ attached to $K$, and it will be
conjectured to remain  non-trivial in settings where ${\rm ord}_{s=1} L(E/K,s) = 2$.

  \subsection{$E(\C)$-valued harmonic cocycles}
  To parlay the system $c_f$ of 
mock residues  attached to $f$ 
into  a rigorous evaluation of the 
 plectic invariant attached to $\tau$, it is natural to replace the $\Gamma$-stable subset $\PP_1(\Q)\subset \cH^*$ of Section 
 \ref{sec:modular-symbols}
  by the $\Gamma$-orbit  
  $$\Sigma := \Gamma \tau_\infty$$
   of $\tau_\infty$ in $\cH$. 
   For  each pair $(x,y)\in \Sigma^2$,  one obtains
 a $\C$-valued harmonic cocycle on $\cT$ via
 \eqref{eqn:C-dist}, denoted $c_f[x,y]$.
 The collection of $c_f[x,y]$ 
 as $x,y$ vary over $\Sigma$  satisfies the $\Gamma$-equivariance property
 $$ c_f[\gamma x,\gamma y](\gamma e) = c_f[x,y](e), \qquad \mbox{ for all } \gamma \in \Gamma, \ x,y\in \Sigma, \mbox{ and } e \in \cE(\cT).$$
 Recall that $\Lambda_f$ is a lattice in $\C$ containing all the periods of the form $2\pi i\int_r^s f(z) dz$ with $r,s \in \PP_1(\Q)$, and that $E$ has been replaced by the isogenous curve with 
 period lattice $\Lambda_f$, which is possible by the Manin-Drinfeld theorem.
  The theory of modular symbols can be invoked  to 
   obtain a 
  $\Gamma$-equivariant collection $\{c_f[x]\}_{x\in \Sigma}$ of harmonic cocycles, indexed by a single $x\in \Sigma$, but with values in $ \C/\Lambda_f = E(\C)$,  satisfying
  $$ c_f[x,y] = c_f[y] - c_f[x] 
  \pmod{\Lambda_f}, 
  \qquad \mbox{ for all } x,y \in \Sigma.$$
  This is done by
setting
\begin{equation}
\label{eqn:transgress-infty}
\begin{split}c_f[x](e) &=  \lvert E(\Q)_\mrm{tors}\rvert\cdot2\pi i\int_{i\infty}^x f_e(z) \hspace{0.5mm}\mathrm{d}z \pmod{\Lambda_f}  \\
&=  \lvert E(\Q)_\mrm{tors}\rvert\cdot2\pi i \int_{i \infty} ^{\gamma x} f(z)\hspace{0.5mm}\mathrm{d}z  \pmod{\Lambda_f}
\end{split}\qquad\mbox{where } \gamma e = e_{\infty} \mbox{  for  } \gamma\in \Gamma.\end{equation}
 As in Section \ref{sec:distributions},
the harmonic cocycle
$c_f[x]$ gives rise to an
 $E(\C)$-valued distribution on 
$\PP_1(\Q_p)$, denoted  $\mu_f[x]$, which can only be integrated against locally constant 
$\Z$-valued functions on $\PP_1(\Q_p)$. 

\subsection{The isogeny tree of a CM curve}\label{IsogenyTree}
The eventual upgrading of $\mu_f[x]$  to a measure is based on the observation that 
the values in 
\eqref{eqn:transgress-infty}
can be interpreted as 
 Heegner points on $E_{/\bb{Q}}$ attached to CM points 
 of $p$-power  conductor on the modular curve $X_0(p)$.
 These points are defined over the anticyclotomic extension
$$ K_\infty = \bigcup_{n=0}^\infty K_n,$$
where $K_n$ is the ring class field of $K$ of conductor $p^n$.
This field is totally ramified at the (unique) prime of $K$ above $p$.
Write
$$ G_n = {\rm Gal}(K_n/K), \qquad G_\infty = {\rm Gal}(K_\infty/K) = \lim_{\leftarrow} G_n.$$
Global class field theory identifies $G_\infty$ with $K_{p,1}^\times$ the group of norm one elements, 
\[
\mrm{rec}\colon K_{p,1}^\times\overset{\sim}{\longrightarrow}G_\infty.
\]
 Let  $A$ be the elliptic curve over $\overline{\Q}$  with complex multiplication by 
the maximal order $\cO_K$. It  is unique up to  isomorphism over $\overline{\Q}$  and has a model
over $\Q$  because of the running class number one assumption. 
Let $\cT_A$ be the $p$-isogeny graph of $A$, whose vertices are elliptic curves over $\overline{\bb{Q}}$ 
related to $A$ by a cyclic isogeny of $p$-power degree, 
and whose edges correspond to $p$-isogenies.
This graph is a tree of valency $(p+1)$ with a 
 distinguished vertex $v_A$ attached to $A$. Since the elliptic curves  that are $p$-power isogenous to $A$ are all defined over $K_\infty$, 
 the  Galois group $G_\infty = K_{p,1}^\times$ acts on $\cT_A$ in the natural way. This action fixes $v_A$ and transitively permutes all the vertices (or edges) that lie at a fixed distance from $v_A$. More precisely, for every $n\ge0$ the subgroup $\scr{U}_n\le K_{p,1}^\times$ attached  to the ring class field $K_n$ under the Galois correspondence
 stabilizes all the vertices of $\cT_A$ at distance $n$ from $v_A$. Choose a sequence of adjacent vertices $\{v_n\}_{n\ge0}$ satisfying
 \[
 \scr{U}_n=\mrm{Stab}_{K_{p,1}^\times}(v_n).
 \]
 For every $n\ge1$ the oriented edge $e_n=(v_{n-1},v_{n})$ from $v_{n-1}$ to $v_{n}$ satisfies
 \[
  \scr{U}_n=\mrm{Stab}_{K_{p,1}^\times}(e_n).
 \]
 The   vertices (resp. edges) of $\cT_A$ at distance $n$ from $v_A$ are in bijection with the $G_n=K^\times_{p,1}/\scr{U}_n$-orbit of $v_n$ (resp. $e_n$).
 It is convenient to interpret each  vertex of $\cT_A$ as a point on the $j$-line $X_0(1)$, and to view
 each edge as a point on the modular curve
 $X_0(p)$, the coarse moduli space of
 pairs of elliptic curves related by a $p$-isogeny.  For $n\ge1$ let $P_n\in X_0(p)(K_{n})$ be the point corresponding to the oriented edge $e_n$.
 Since $\scr{U}_n/\scr{U}_{n+1}$ acts simply transitively on the set of edges at distance $n+1$ from $v_A$ having $v_n$ as an endpoint, 
   it follows  that
 \[
 \mrm{Tr}_{K_{n+1}/K_n}(P_{n+1})= U_p(P_n) \qquad\forall\ n\ge1.
 \]
 
 \begin{remark} \label{Trace-compatible}
 
   Recall we previously defined $a_p(E)= 1$  or $-1$ depending on whether $E$ has split or non-split multiplicative reduction at $p$. If we write $y_n\in E(K_n)$ for the Heegner point
 arising from the divisor $a_p(E)^n\cdot (P_n-\infty)$ on  $X_0(p)$  through the modular parametrization $\varphi_E\colon X_0(p)\to E$ (normalized by $\varphi_E(\infty)=0_E$), then the collection $\{y_n\in E(K_n)\}_{n\ge1}$ is trace-compatible. 
 \end{remark}

 Fix a trivialisation $H_1(A(\C),\Z_p) \simeq \Z_p^2$. 
 The choice of a complex embedding 
 $\iota_\infty: K_\infty \hookrightarrow \C$ together with Shimura's reciprocity law
 determine a $K_{p,1}^\times$-equivariant graph isomorphism 
 $$ j_A: \cT_A\overset{\sim}{\lra} \cT$$
 which sends the vertex attached 
 to $A'$ to the lattice $ H_1(A'(\C),\Z_p) \subseteq H_1(A(\C),\Q_p) = \Q_p^2$, after viewing $A'$ as a curve over
 $\C$ via $\iota_\infty$. In particular, $j_A$ maps the distinguished vertex $v_A$ to $v_\circ$. 
 The
 identification $j_A$  allows the harmonic 
 cocycle $c_f[\tau_\infty]$    to be
 viewed as taking values in 
 ${E(K_\infty)}$. 
 More precisely, \cite[Lemma 1.5]{darmon-hpxh} and equation \eqref{eqn:transgress-infty}  give
 \begin{equation}\label{keyidentity}
 c_f[\tau_\infty](\alpha\cdot e_n)= y_n^{\mathrm{rec}(\alpha)},\qquad \forall\ n\ge1,\ \alpha\in K_{p,1}^\times.
 \end{equation}
The general case of $\tau=\gamma^{-1}\tau_\infty\in\Sigma$ is dealt with by the formula
    \[
    c_f[\tau](e)=c_f[\tau_\infty](\gamma e)\qquad\forall\ e\in\mathcal{E}(\cT).
    \]

\subsection{Measures and the Poisson transform}
In order to integrate continuous function with respect to the measure $\mu_f[\tau_\infty]$, the value group $E(K_\infty)$ needs to be
 $p$-adic  completed. We will denote the $p$-adic completion of the  (infinitely generated)  Mordell-Weil group $E(K_\infty)$ by
$$\widehat{E(K_\infty)} := \varprojlim_n\hspace{1mm} E(K_\infty) \otimes \Z/p^n\Z. $$ 
Viewing $\mu_f[x]$ (for $x\in \Sigma$) 
as an $\widehat{E(K_\infty)}$-valued measure on $\PP_1(\Q_p)$,
 the Teitelbaum transform  of $\mu_f[x]$ gives a collection of elements 
\begin{equation}
\label{eqn:hp-ppt}
 \omega_f[x] := \int_{\PP_1(\Q_p)} \frac{\mathrm{d}z}{z-t}\hspace{1mm} \mathrm{d}\mu_f[x](t) \ \ \in \ \ 
\Omega^1_{\mathrm{ rig}}(\cH_p)  {\widehat \otimes} {\widehat{ E(K_\infty)}}.
\end{equation}
For any $x\in \Sigma$, let 
$$ \iota_x: K \longrightarrow M_2(\Q)$$
be the algebra embedding that sends $K$ to the fraction field of the order $\cO_x$.
The group $\Gamma$ acts on $\Omega^1_{\mathrm{ rig}}(\cH_p)  {\widehat \otimes} {\widehat{ E(K_\infty)}}$ by translation on $\cH_p$, and the Galois group  $ G_\infty$ acts via its natural action on ${\widehat{ E(K_\infty)}}$.
\begin{proposition}
\label{prop:properties-omegatau}
The ${\widehat{ E(K_\infty)}}$-valued rigid differentials $\omega_f[x]$ satisfy the following properties:
\begin{itemize}
\item [$\bfcdot$] For all $\gamma\in \Gamma$, $x\in \Sigma$
$$ \gamma^* \omega_f[\gamma x] = \omega_f[x].$$
\item[$\bfcdot$] For all $\alpha\in K_{p,1}^\times$, 
$$ \iota_x(\alpha)^\ast (\omega_f[x]) = 
{\mathrm {rec}}(\alpha) \omega_f[x].$$
\end{itemize}
\end{proposition}
The second part of this proposition is particularly noteworthy: setting $x= \tau_\infty$,
it relates the action of the $p$-adic 
torus $\iota_\tau(K_p^\times)$ on $\cH_p$, which fixes $\tau_p$, to the  Galois action on the elements
 $\omega_f[\tau] \in \Omega^1_{\rm rig} {\hat\otimes} \widehat{E(K_\infty)}$,   and is just a reformulation of the Shimura reciprocity law.

 \subsection{The mock plectic invariant}
 Following the same ideas as in the proof of
 Lemma \ref{lemma:invert-dlog}, 
a well-defined system  of multiplicative primitives 
 $$ F_f[x] \in (\cA^\times/K_p^\times) {\hat \otimes}\widehat{E(K_\infty)}$$
 can be attached to the elements $\omega_f[x]$,
 satisfying 
 $$ {\rm dlog}(F_f[x]) = \omega_f[x] \qquad \mbox{for all } x\in \Sigma.$$
The torus $\iota_\tau(K_p^\times)$ has {\em two} fixed points $\tau_p$, $\overline{\tau}_p$ acting on $\cH_p$, they  are interchanged by the action of ${\mathrm {Gal}}(K_p/\Q_p)$.  
This   circumstance leads to the definition of the multiplicative
Nekov\'a$\check{\text{r}}$--Scholl
 {\em mock plectic invariant} attached to the CM point $\tau$, 
 by setting
  $$\NS^\times(\tau) := \frac{F_f[\tau_\infty](\tau_p)}{F_f[\tau_\infty](\overline{\tau}_p)} \ \  \in \ \  \widehat{E(K_\infty)} {\widehat \otimes} K_{p,1}^\times. $$
  Since $p$ is inert in $K$, the group $K_{p,1}^\times$ 
  consists of $p$-adic units and the invariant $\NS^\times(\tau)$ is almost completely determined by its $p$-adic logarithm
  $$ 
 \NS(\tau) := \log \NS^\times(\tau)=  \int_{\overline{\tau}_p}^{\tau_p} \!\omega_f[\tau_\infty] \ \ \in \ \  \widehat{E(K_\infty)} {\widehat \otimes} K_{p}.$$
\begin{lemma}
\label{lemma:ns-inv-inv}
The mock plectic invariant $\NS(\tau)$ belongs to 
$\big(\widehat{E(K_\infty)} {\widehat \otimes} K_{p}\big)^{G_\infty}$.
\end{lemma}
\begin{proof}
For all $\alpha\in K_{p,1}^\times$, we have
$$ {\mathrm {rec}}(\alpha) \NS(\tau) = \int_{\overline{\tau}_p}^{\tau_p}  {\mathrm {rec}}(\alpha) \omega_f[\tau] = \int_{\overline{\tau}_p}^{\tau_p}  \iota_\tau(\alpha)^\ast \omega_f[\tau] = \NS(\tau),$$
where the penultimate equality follows from 
the second assertion in Proposition
\ref{prop:properties-omegatau}, and the last from the change of variables formula and the fact that 
$\iota_\tau(K_p^\times)$ fixes both $\tau_p$ and $\overline{\tau}_p$.
\end{proof}

\subsection{Anticyclotomic $p$-adic $L$-functions}
We will now give a formula for the mock plectic invariant
$\NS(\tau)$ in terms of the first derivatives of certain ``anticyclotomic $p$-adic $L$-functions'' in the sense of \cite[Section 2.7]{bd-Mumford}.

\smallskip
Recall that the CM point $\tau=(\tau_p,\tau_\infty)$ determines an embedding $K_p\subseteq\mrm{M}_2(\bb{Q}_p)$ and hence an action of $K_p^\times$ on $\bb{P}_1(\bb{Q}_p)$. Let 
\begin{equation}
\label{homeo}
 A\colon \PP_1(\Q_p) \lra K_{p,1}^\times, 
\qquad A(x) = \frac{x-\tau_p}{x-{\overline \tau}_p}
\end{equation}
be the  M\"obius transformation that sends $(\tau_p,\bar\tau_p,\infty)$ to $(0,\infty,1)$,
and
let $\mu_{f,K}$ be the pushforward of the measure $\mu_f[\tau_\infty]$ to $K_{p,1}^\times$ via $A$:
$$ \mu_{f,K} := A_\ast \mu_f[\tau_\infty].$$
The {\em anticyclotomic $p$-adic $L$-function}
attached to $(E,K)$ is the Mazur-Mellin transform of the measure $\mu_{f,K}$:
\begin{equation}
\label{eqn:def-LpEK}
 L_p(E,K,s) := \int_{K_{p,1}^\times} \langle \alpha\rangle^{s-1} d\mu_{f,K}(\alpha).
 \end{equation}
It can be viewed as  a $p$-adic analytic function from $1+p\Z_p$ to  ${\widehat{E(K_\infty)}} \widehat\otimes K_p$. 
\begin{theorem}
\label{prop:BD-old}
The $p$-adic $L$-function $L_p(E,K,s)$ vanishes at $s=1$ and 
$$ \NS(\tau) = L_p'(E,K,1).$$
\end{theorem}
\begin{proof}
The vanishing of $L_p(E,K,1)$ follows from the fact that $\mu_{f}[\tau_\infty]$, and hence also 
$\mu_{f,K}$, have total measure zero.
By the definition of $\NS(\tau)$ combined with 
\eqref{eqn:hp-ppt},
$$ \NS(\tau) = \int_{\bar \tau_p}^{\tau_p} \omega_f[\tau_\infty] =  \int_{\bar \tau_p}^{\tau_p}
\left(\int_{\PP_1(\Q_p)} \frac{1}{z-t} d\mu_f[\tau_\infty](t) \right) dz.$$
Interchanging the order of integration  and integrating with respect to $z$
gives
$$ \NS(\tau) =  \int_{\PP_1(\Q_p)}
\log \left( \frac{\tau_p-t}{\bar\tau_p-t}\right) d\mu_f[\tau_\infty](t) = \int_{\PP_1(\Q_p)}
\log A(t) \cdot d\mu_f[\tau_\infty](t). $$
The change of variables $\alpha=A(t)$ 
can be used to rewrite this last expression as an integral over $K_{p,1}^\times$,
$$ \NS(\tau) = \int_{K_{p,1}^\times} \log(\langle\alpha \rangle) \cdot d\mu_{f,K}(\alpha) = L_p'(E,K,1),$$
where the last equality follows  directly from 
\eqref{eqn:def-LpEK}.
\end{proof}
  As explained in \cite{bd-rigid}, 
  the quantity $L_p'(E,K,1)$ is directly related to the {\em Kolyvagin derivative}
  of the norm-compatible collection $y_n= c_f[\tau_\infty](e_n)\in E(K_n)$ of Heegner points, 
 where  $\{e_n\}_{n\ge1}\subset\mathcal{E}(\cT)$ 
is the  sequence of adjacent edges
that was fixed  in Section \ref{IsogenyTree}. 
 More precisely, for every $n\ge1$, the collection $\{\alpha\cdot U_{e_n}\}_{\alpha\in K_{p,1}^\times}$  is a finite   covering of $\PP_1(\Q_p)$ 
 by compact open subsets, 
 which are equal to the cosets $\{\alpha\cdot \mathscr{U}_n\}_{\alpha\in K_{p,1}^\times}$ under the homeomorphism \eqref{homeo}.
 Rewriting $L_p'(E,K,1)$ as the limit when $n\rightarrow \infty$ of  the Riemann sums attached to these coverings,
  equation \eqref{keyidentity} gives 
  \begin{equation}
\label{L-as-Kol}
\begin{split}
\NS(\tau) &=\int_{K_{p,1}^\times} \log(\alpha)  d\mu_{f,K}(\alpha)\\
&=\underset{n\rightarrow\infty}{\lim}\ \sum_{\alpha\in K_{p,1}^\times/\mathscr{U}_n}c_f[\tau_\infty](\alpha\cdot e_n)\otimes \log(\langle \alpha\rangle)\\
&=\underset{n\rightarrow \infty}{\lim}\ \sum_{\alpha\in K_{p,1}^\times/\mathscr{U}_n}y_n^{\mrm{rec}(\alpha)} \otimes\log(\langle\alpha\rangle)
\end{split}
\end{equation}
(see \cite[\S 6]{CDuniformization} for related discussions).

\subsection{Kolyvagin's cohomology classes}
After choosing a topological generator of $K_{p,1}^\times$, Lemma \ref{lemma:ns-inv-inv} allows the mock plectic
invariant  $\NS(\tau)$ to be viewed (non-canonically)
as an element of $\widehat{E(K_\infty)}$ fixed by $G_\infty$.
Consider the natural injective map
$$ E(K)\otimes \Z_p \longrightarrow \widehat{E(K_\infty)}^{G_\infty}.$$ 
It need not be  surjective in general: the group
$\widehat{E(K_\infty)}$ fails to   satisfy the principle of Galois descent.
Moreover, while it is relatively straightforward to establish the infinitude of $\widehat{E(K_\infty)}^{G_\infty}$, the Birch and Swinnerton-Dyer conjecture predicts the finitude of $E(K)$ when $L(E/K,1)\ne 0$. Even when $L(E/K,s)$ vanishes at $s=1$, and hence has a zero of order $\ge 2$, 
 establishing that $E(K)\otimes \Z_p$ is  infinite
seems very hard to  do  unconditionally.

A useful handle on the group $\widehat{E(K_\infty)}^{G_\infty}$
is obtained by relating it to  Galois
 cohomology.
 Let $H^1(K_m, E[p^n])$ be the first Galois cohomology of $K_m$ with values in the module of $p^n$-division points of $E$, and let 
 $$ H^1(K_m, T_p(E)) = \varprojlim_n\hspace{1mm} H^1(K_m, E[p^n]),$$
 the inverse limit being taken relative to the multiplication by $p$ maps $E[p^{n+1}] \lra E[p^n]$. 
 \begin{lemma} 
 \label{lemma:p-tors-bounded}
 The module  $E[p^n](K_m)$  is trivial for all  $m$ and $n$.
 \end{lemma}
 \begin{proof}
 Since $E$ is semistable of prime conductor $p\ge 11$,
 its mod $p$ Galois representation is surjective (\cite{Mazur-Goldfeld}, Theorem 4). 
It follows that the mod $p$ Galois representation has non-solvable image and therefore $E[p](L)$ is trivial for any solvable extension $L/\Q$. 
 \end{proof}
 
\begin{lemma}
\label{lemma:res-iso}
The restriction map
$$ H^1(K, E[p^n]) \lra H^1(K_m, E[p^n])^{G_m}$$
is an isomorphism.
\end{lemma}
\begin{proof}
The kernel and cokernel in
the  inflation-restriction sequence
$$
\xymatrix{ 
H^1(G_m,E[p^n](K_m)) \ar[r]&    H^1(K, E[p^n]) \ar[r]^{\!\!\! \!\!\!\!\!\!  \rm res}&    H^1(K_{m}, E[p^n])^{G_{m}}\ar[r]& 
H^2(G_m,E[p^n](K_m))
}
$$
  are trivial by  Lemma
 \ref{lemma:p-tors-bounded},
 and the claim follows.
\end{proof}
Denote by $\delta_n$ the mod $p^n$ Kummer map
$$ \delta_n: E(K_m)/p^n E(K_m) \lra 
H^1(K_m, E[p^n]),$$
and by 
$$ \delta_\infty: E(K_m)\otimes \Z_p   \lra 
H^1(K_m, T_p(E))$$
the map induced on the inverse limits.
Let $P_\infty = \{ P_n \}_{n\ge 1} \in {\widehat{E(K_\infty)}}^{G_\infty}$, 
where 
$$P_n \in (E(K_{m_n})\otimes \Z/p^n\Z)^{G_{m_n}} $$ for   suitable integers $m_n\ge 1$. Lemma \ref{lemma:res-iso}
shows that each $\delta_n(P_n)\in H^1(K_{m_n}, E[p^n])^{G_{m_n}}$ 
is the restriction of a unique class $\kappa_n \in H^1(K, E[p^n])$. The $\kappa_n$
 are compatible under the natural multiplication by $p$ maps from $E[p^{n+1}]$ to $E[p^n]$, and the assignment $(P_n)_{n\ge 1} \mapsto (\kappa_n)_{n\ge 1}$ determines
 a canonical inclusion
$$    \widehat{E(K_\infty)}^{G_\infty} \lhook\joinrel\longrightarrow H^1(K,T_p(E)).$$
It will be convenient for the rest of this
note  to identify 
 $\NS(\tau)$ with its image in $H^1(K,T_p(E))$
 under this map.

 \subsection{Local properties of 
 the mock plectic invariant}
For each place $v$ of $K$, let 
$$\res_v: H^1(K,T_p(E)) \lra H^1(K_v, T_p(E))$$ be the local restriction map to a decomposition group at $v$.  
The local cohomology group $H^1(K_p,T_p(E))$ is equipped with a natural two-step filtration
$$ 0 \lra  H^1_f(K_p, T_p(E)) {\lra} H^1(K_p,T_p(E)) \lra 
H^1_{\rm sing}(K_p, T_p(E)) \lra 0,$$
where $H^1_f(K_p, T_p(E)):= \delta_\infty(E(K_p) \otimes \Z_p)$.  
\begin{definition}
The {\em pro-$p$-Selmer group} of $E$
is the group, denoted $H^1_f(K, T_p(E))$, of global classes $\kappa\in H^1(K,T_p(E))$ satisfying
$$\res_p(\kappa) \in H^1_f(K_p, T_p(E)).$$ 
 \end{definition}
The reader will note that in the setting of pro-$p$ Selmer groups, no local conditions need to be imposed at the places $v\ne p$ because the Kummer map $\delta_{v}: E(K_v) \otimes \Z_p \lra H^1(K_v, T_p(E))$ is an isomorphism.
The only possible obstruction for a global class to lie in the pro-$p$ Selmer group therefore occurs at the place $p$.
Let 
 $$\NS_p(\tau):= {\rm res}_p(\NS(\tau))  \in H^1(K_p,T_p(E)) $$ 
 denote the restriction of  the global class 
 $\NS(\tau)$ to the decomposition group at $p$.
The first important result about the mock plectic invariant is that it is non-Selmer 
precisely when $L(E/K,1) \ne 0$.
\begin{theorem}
 \label{SelmerClass}
The  class $\NS_p(\tau)$  lies in $H^1_f(K_p, T_p(E))$  if and only if $L(E/K,1)= 0$.
\end{theorem}
\begin{proof}[Sketch of proof]  
Recall that the ring class field $K_n$ of $K$ of conductor $p^n$ is a cyclic Galois extension of $K$ with Galois group $G_n\cong K_{p,1}^\times/\mathscr{U}_n$. The Kolyvagin derivative of the Heegner point $y_n\in E(K_n)$ is defined as 
\[
D_ny_n:=\sum_{\alpha\in K_{p,1}^\times/\mathscr{U}_n}y_n^{\mrm{rec}(\alpha)}\otimes\alpha\ \ \in \ \ E(K_n)\otimes K_{p,1}^\times/\mathscr{U}_n
\]
and is fixed by the action of $G_n$ because $\mrm{Tr}_{K_n/K}(y_n)=0$ (cf.~\cite[(8)]{bd-Mumford} and  \cite[Proposition 3.6]{gross}). Then, just as in equation \eqref{L-as-Kol} one sees that
\[
\NS^\times(\tau)  =\varprojlim_n\hspace{1mm} D_ny_n.
\]
By choosing a topological generator of $K_{p,1}^\times$ we can consider the image of $\NS^\times(\tau)$ in $H^1(K,T_p(E))$ and study its local properties via the following   diagram taken from \cite[(4.2)]{gross}:
\[\xymatrix{
&&&0\ar[d]\\
&&0\ar[d]&H^1(G_n,E(K_n))\ar[d]^{\mrm{Inf}}\\
E(K_n)^{\mrm{Tr}_{K_n/K}=0}\ar[rd]^{D_n}& &H^1(K,E[p^{n-1}])\ar[d]^{\mrm{Res}}\ar[r]&H^1(K,E)\ar[d]^{\mrm{Res}}\\
0\ar[r]&\big(E(K_n)/p^{n-1}\big)^{G_n}\ar[r]^-\delta&H^1(K_n,E[p^{n-1}])^{G_n}\ar[d]\ar[r]&H^1(K_n,E)^{G_n}\\
&&0&.
}\]
By construction, the image of $D_ny_n$ in $H^1(K,E)$ belongs to the image of $H^1(G_n,E(K_n))$ under inflation and  it can be represented by the following $1$-cocycle \cite[Lemma 4.1]{mccallum}:
\begin{equation}\label{explicitMcCallum}
G_n\ni\sigma\mapsto-\frac{(\sigma-1)D_n(y_n)}{p^{n-1}},
\end{equation}
where $\frac{(\sigma-1)D_n(y_n)}{p^{n-1}}$ denotes the unique $p^{n-1}$-th root of $(\sigma-1)D_n(y_n)$ in $E(K_n)$. Now, we can use the explicit description \eqref{explicitMcCallum} to study the image $\partial_p\NS^\times(\tau)$ of $\NS^\times(\tau)$ in  $H^1_\mrm{sing}(K_p,T_p(E))$.

\medskip
\noindent Denote by $\Phi_n$ the group of connected components of the special fiber of the N\'eron model of $E$ over the $p$-adic completion $K_{n,p}$ of $K_n$. By \cite[Lemma 6.7]{bd-rigid} the reduction map $E(K_{n,p})\to \Phi_n$ produces  an injection
    \[
    H^1(G_n, E(K_{n,p}))[p^{n-1}] \lhook\joinrel\longrightarrow \Phi_n[p^{n-1}],
    \]
    where $H^1(G_n,\Phi_n)=\mrm{Hom}_\Z(G_n,\Phi_n)$ is identified with $\Phi_n$ by evaluation at a generator of $G_n$. A direct computation using \eqref{explicitMcCallum} and \cite[(3.5)]{gross} then shows that $ \partial_p\NS^\times(\tau) $  vanishes exactly when the compatible collection of Heegner points $\{y_n\}_{n\ge1}$ maps to a finite order element in the projective limit $\{\Phi_n\}_{n\ge1}$.
By adapting the proof of \cite[Theorem 5.1]{bd-rigid} one sees that this happens precisely when the $L$-value $L(E/K,1)$ vanishes.  
\end{proof}

Interestingly, when $L(E/K,1)\ne 0$, the mock plectic invariant $\NS(\tau)$ suffices to prove that the Mordell-Weil group of $E'$ -- the quadratic twist of $E$ attached to $K$ -- is finite.

\begin{theorem}\label{application-nonSelmer}
If $L(E/K,1)\ne 0$, then $E'(\bb{Q})$ is finite.
\end{theorem}
\begin{proof}[Sketch of proof]
We follow the proof of \cite[Corollary 7.2]{bd-rigid}. By Theorem \ref{SelmerClass} we know that $\NS(\tau)$ is a global class ramified only at $p$. The claim is then that the localization map $E'(\bb{Q})\otimes\bb{Q}_p\hookrightarrow E'(\bb{Q}_p)\otimes\bb{Q}_p$ -- which is always injective -- is the zero morphism. 

\smallskip
 Under our assumptions, the hypothesis $L(E/K,1)\not=0$ also implies that $E$ has split multiplicative reduction at $p$, hence the class $\NS(\tau)$ lives in the minus-eigenspace for the action of complex conjugation \cite[Prop.~6.5]{bd-rigid}  and its image in $H^1_{\rm sing}(K_p, V_p(E))$ satisfies $H^1_{\rm sing}(\bb{Q}_p, V_p(E'))=\bb{Q}_p\cdot \NS_p(\tau)$.  The claim then follows because local Tate duality induces the identification  
 \[
 E'(\bb{Q}_p)\otimes\bb{Q}_p\overset{\sim}{\longrightarrow}H^1_{\rm sing}(\bb{Q}_p, V_p(E'))^\vee\qquad Q\mapsto \big\langle Q, -\big\rangle_p,
 \]  
 and Poitou--Tate duality implies that any point $P\in E(K)\otimes\bb{Q}_p$ satisfies
\[
0=\sum_\ell\big\langle \mrm{res}_\ell(P), {\rm res}_\ell (\NS(\tau))\big\rangle_\ell=\big\langle \mrm{res}_p(P), \NS_p(\tau)\big\rangle_p.
\]
(See \cite[Prop.~6.8]{bd-rigid}.)
\end{proof}

\begin{remark}
   To obtain the finiteness of $E(K)$ from $L(E/K,1)\ne 0$ one also needs to consider tame deformations of mock plectic invariants (cf.~\cite[Rem.~p.132]{bd-rigid}). 
\end{remark}

\subsection{Elliptic curves of rank two}
When $L(E/K,1)=0$, the local class 
$\NS_p(\tau)$ belongs to $H^1_f(K_p,T_p(E))\otimes K_p$, 
i.e.,  the global class $\NS(\tau)$  lies in the pro-$p$ Selmer group of $E$ over $K$. 
It is not expected to be trivial in general: in fact, as we now proceed to explain, it should  provide  a non-trivial Selmer class  in settings where  
 $L(E/K,s) = L(E,s) L(E',s)$ has a double zero at $s=1$, and even when  the   factor 
 $L(E',s)$ admits such a double zero.

\smallskip
Since the \emph{first} derivative of the $p$-adic $L$-function $L_p(E,K,s)$ computes the mock plectic invariant $\NS(\tau)$ according to Theorem \ref{prop:BD-old}, we can predict when $L(E/K,1)=0$ and $\NS(\tau)$ does not vanish by analyzing the  anticyclotomic Birch and Swinnerton-Dyer conjecture of \cite[Conj.\hspace{-1mm} 4.1]{bd-Mumford}. 
In the terminology of \cite[Introduction $\&$ Definition 2.3]{bd-Mumford}, the triple $(E,K,p)$ falls into the non-split exceptional indefinite case, where the adjective \emph{non-split exceptional} refers to the fact that the elliptic curve has multiplicative reduction at the prime $p$ inert in $K$ (and not to the fact that  $E$  has necessarily  non-split multiplicative reduction at $p$). If we set $E^+ :=E$, denote by $E^- = E'$ the quadratic twist of $E$ attached to $K$,  and define
\[
\delta^\pm=\begin{cases}
    1&\text{if}\ a_p(E^\pm)=+1\\
    0&\text{if}\ a_p(E^\pm)=-1,
    \end{cases}
\]
then \cite[Conjecture 4.1]{bd-Mumford} suggests that  the order of vanishing of  $L_p(E,K,s)$ at $s=1$ is
\begin{equation}\label{orderVAN}
\varrho=\mrm{max}\big\{1,\hspace{1mm} r_\mrm{alg}(E^\pm/\bb{Q})+\delta^\pm -1\big\},
\end{equation}
which satisfies $2\varrho\ge \mrm{max}\big\{2,\hspace{1mm} r_\mrm{alg}(E/K)\big\}$. 
\begin{remark}
    The $p$-adic $L$-function $L_p(E,K,s)$ can be obtained from the $p$-adic $L$-function considered in \cite[Conj.\hspace{-2mm} 4.1]{bd-Mumford} by collapsing the extended Mordell--Weil group to the usual Mordell--Weil group. This explains the discrepancy between \eqref{orderVAN} and \cite[Conj.\hspace{-1mm} 4.1]{bd-Mumford}.
\end{remark}

\noindent We deduce that, when $L(E/K,1)=0$, we should expect $\NS(\tau)$ to be non-trivial only when $r_\mrm{alg}(E/K)=2$ and the rank is distributed in the two eigenspaces for the action of complex conjugation according to a rule depending on reduction type at the prime $p$: 
\begin{conjecture}\label{Conj1}
  Suppose that   $L(E/K,1)= 0$, then
  \[
  \NS(\tau)\not=0\quad\iff\quad   \begin{cases}
      r_\mrm{alg}(E/\bb{Q})=0, \quad  r_{\rm alg}(E'/\Q)=2\quad  &\mathrm{ if}\hspace{2mm} a_p(E)=+1,\\
       r_\mrm{alg}(E/\bb{Q})=1,\quad r_{\rm alg}(E'/\Q)=1\quad &\mathrm{ if}\hspace{2mm} a_p(E)=-1.
  \end{cases}
  \]
\end{conjecture}

In order to make explicit the relation between the mock plectic invariant $ \NS(\tau)$ and global points in $E(K)$, denote by  $\log_{E}^{a_p}\colon E(K_p)\to K_p$ the composition of the $p$-adic logarithm of $E$ with the endomorphism $(1-a_p(E)\cdot\sigma_p)$ of $E(K_p)$, where $\sigma_p\in\text{Gal}(K_p/\bb{Q}_p)$ denotes the non-trivial involution. 
\begin{conjecture}\label{Conj2}
  Suppose that   $L(E/K,1)= 0$. Then
  $\NS(\tau)$ is in the image of the regulator $ \wedge^2E(K)\to H^1_f(K,T_p(E))\otimes_{\Z_p} K_p$  given by
    \[
P\wedge Q\mapsto \delta_\infty(P)\otimes \log_{E}^{a_p}(Q)-\delta_\infty(Q)\otimes\log_{E}^{a_p}(P).
  \]
\end{conjecture}

\begin{remark}
    Given the analogy between mock plectic invariants and the plectic $p$-adic invariants of \cite{fgm}, \cite{fornea-gehrmann},  the reader is invited to compare Conjectures \ref{Conj1} and \ref{Conj2} with \cite[Conjectures 1.5 $\&$ 1.3]{fgm}. 
\end{remark}

The fact that the plectic invariant is forced to lie in a specific eigenspace for
complex conjugation and can sometimes vanish for trivial reasons (for example, when $E(\Q)$ has rank two) suggests that  it is only ``part of the story" and represents the projection of a more complete invariant which should be non-trivial in all scenarios when $E(K)$ has rank two. 
The authors believe that a full  mock analogue  of plectic Stark--Heegner points 
can be obtained by exploiting the $p$-adic deformations of $f$ arising 
from  Hida theory, and hope to treat this idea in future work. These ``mock plectic points'' should   control the arithmetic of elliptic curves of rank two over quadratic imaginary fields, and be unaffected by the degeneracies of anticyclotomic height pairings that plague \cite[Conjecture 4.1]{bd-Mumford}.

\medskip
The approach of this article might also be exploited to upgrade the  plectic Heegner points of \cite{fornea-gehrmann} from tensor products of $p$-adic points  to global cohomology classes. Such an improvement would take care of the degeneracies of the original construction (evoked, for example, in \cite[Remark 1.1]{fornea-gehrmann}) and, in light of Theorems \ref{SelmerClass} and \ref{application-nonSelmer}, it would help explain the arithmetic meaning of plectic Stark--Heegner points in the setting considered in \cite[Remark 1.3]{fornea-gehrmann}. (See also the paragraph before \cite[Conjecture 1.6]{fgm}.)


\end{document}